\documentclass[twoside,11pt,reqno]{amsart}
\usepackage{amsmath,amssymb,amscd,mathrsfs}

\makeatletter

\@addtoreset{equation}{section}

\input prepictex
\input pictex
\input postpictex

\newtheorem{Proposition}{Proposition}[section]
\newtheorem{Lemma}[Proposition]{Lemma}
\newtheorem{Theorem}[Proposition]{Theorem}
\newtheorem{Corollary}[Proposition]{Corollary}
\newtheorem{Definition}[Proposition]{Definition}

\newbox\squ  
\setbox\squ=\hbox{\vrule width.6pt
   \vbox{\hrule height.6pt width.4em\kern1ex
         \hrule height.6pt}%
   \vrule width.6pt}

\def\mod#1{#1\!\operatorname{-mod}}

\def\Z{{\mathbb Z}}

\def\0{{\bar 0}}
\def\1{{\bar 1}}

\def\R{{\mathbb R}}

\def\Res{{\operatorname{res}\:}}
\def\gdim{{\operatorname{gdim}\,}}
\def\sh{{\operatorname{sh}}}

\def\ch{{\operatorname{ch}\:}}

\def\height{{\operatorname{ht}}}

\def\sh{{\operatorname{con}}}

\def\bi{\text{\boldmath$i$}}
\def\bj{\text{\boldmath$j$}}

\def\phi{{\varphi}}

\def\Ga{{\Gamma}}
\def\la{{\lambda}}
\def\La{{\Lambda}}
\def\de{{\delta}}

\def\al{{\alpha}}
\def\be{{\beta}}

\def\underbar{\mathpalette\@underbar}
\def\@underbar#1#2{\settowidth{\@tempdimb}{$#1#2$}\@tempdimb=0.8\@tempdimb
                   \ooalign{$#1#2$\crcr%
                         \hfil\rule[-.5mm]{\@tempdimb}{.4pt}\hfil}}

\newdimen\hoogte    \hoogte=11.5pt    
\newdimen\breedte   \breedte=11.5pt   
\newdimen\dikte     \dikte=0.5pt    

\newenvironment{young}{\begingroup
       \def\vr{\vrule height0.8\hoogte width\dikte depth 0.2\hoogte}
       \def\fbox##1{\vbox{\offinterlineskip
                    \hrule height\dikte
                    \hbox to \breedte{\vr\hfill##1\hfill\vr}
                    \hrule height\dikte}}
       \vbox\bgroup \offinterlineskip \tabskip=-\dikte \lineskip=-\dikte
            \halign\bgroup &\fbox{##\unskip}\unskip  \crcr }
       {\egroup\egroup\endgroup}
\def\diagram#1{\relax\ifmmode\vcenter{\,\begin{young}#1\end{young}\,}\else%
              $\vcenter{\,\begin{young}#1\end{young}\,}$\fi}

\begin{document}

\title[Representations of Khovanov-Lauda algebras]{Homogeneous Representations of Khovanov-Lauda Algebras}
\author{Alexander Kleshchev and Arun Ram}

\begin{abstract}
We construct irreducible graded representations of simply laced Khovanov-Lauda algebras which are concentrated in one degree. The underlying combinatorics of skew shapes and standard tableaux corresponding to arbitrary simply laced types has been developed previously by Peterson, Proctor and Stembridge. In particular, the Peterson-Proctor hook formula gives dimensions of the homogeneous irreducible modules corresponding to straight shapes. 
\end{abstract}
\thanks{Supported in part by NSF grants 
DMS-0654147 and DMS-0353038.}
\address{Department of Mathematics, University of Oregon, Eugene, OR 97403, USA.}
\email{klesh@uoregon.edu}
\address{Department of Mathematics and Statistics,  
The University of Melbourne, 
Parkville, VIC, 3010,  
Australia, and Department of Mathematics, University of Wisconsin, Madson, WI 53706, USA.}
\email{A.Ram@ms.unimelb.edu.au}
\maketitle

\section{Introduction}\label{SIntro}

In \cite{KL1,KL2}, Khovanov and Lauda have introduced a new family of 
graded algebras whose representation theory is related to categorification of quantum groups. Similar algebras have been defined by Rouquier \cite{Ro}. 

In this note we give an explicit construction of 
the irreducible graded representations of simply laced Khovanov-Lauda algebras which are concentrated in one degree. These {\em homogeneous} representations turn out to be similar to  seminormal representations of affine Hecke algebras. In type $A$ this can be explained using \cite{BK}. 

By-products of our construction are notions of skew shape and standard tableaux for arbitrary simply laced types. Equivalent notions have been considered before by Peterson, Proctor, Stembridge, and Fan \cite{P1,P2,S1,S2,F,N1,N2}. In particular, the Peterson-Proctor hook formula gives dimensions of the homogeneous irreducible modules corresponding to straight shapes. 

\subsection*{Acknowledgements}
The first author is grateful to the University of Melbourne for support and hospitality. Both authors are grateful to K. Nakada, J. Stembridge and R. Green for useful comments. 

\section{Khovanov-Lauda algebras}
\subsection{Definition}
Let $\Gamma$ be a graph without multiple edges and loops (cycles allowed). Denote the set of vertices of $\Gamma$ by $I$. If $i,j\in I$ are connected by an edge, we will say that $i$ and $j$ are {\em neighbors} (in $\Ga$). We allow for $I$ to be infinite and for $\Gamma$ to contain cycles. To $\Gamma$ we associate a generalized Cartan matrix $(a_{ij})_{i,j\in I}$ as in \cite{Kac}, so that 
$$
a_{ij}=
\left\{
\begin{array}{ll}
2 &\hbox{if $i=j$,}\\
-1 &\hbox{if $i$ and $j$ are neighbors,}\\
0 & \hbox{otherwise.}
\end{array}
\right.
$$
We fix an orientation on the edges of $\Gamma$. 

Let $Q =\bigoplus_{i \in I} \Z\al_i$ be a lattice with 
a basis $\{\al_i\}_{i\in I}$ labeled by~$I$. Set $$Q_+ =\bigoplus_{i\in I}\Z_{\geq 0}\al_i.$$ For $\al=\sum_{i\in I}m_i\al_i\in Q_+$ 
define the {\em height} of $\al$ as 
$$
\height(\al):=\sum_{i\in I}m_i.
$$


The symmetric group $S_d$ 
with basic transpositions $s_1,\dots,s_{d-1}$ 
acts  
on $I^d$ on the left by place
permutations. 
We have a decomposition of $I^d$ into $S_d$-orbits:
$$
I^d = \bigsqcup_{\alpha\in Q^+\atop \height(\alpha) = d} I^\alpha,
$$
where 
\begin{equation*}
I^\alpha := \{\bi=(i_1,\dots,i_d) \in I^d\:|\:\alpha_{i_1}+\cdots+\alpha_{i_d} = \alpha\}.
\end{equation*}


Fix an arbitrary ground field $F$ and an element $\al\in Q_+$ of height $d$. The {\em Khovanov-Lauda algebra} $R_\al$ is an associative $\Z$-graded unital $F$-algebra, given by generators
\begin{equation}\label{EKLGens}
\{e(\bi)\mid \bi\in I^\al\}\cup\{y_1,\dots,y_{d}\}\cup\{\psi_1, \dots,\psi_{d-1}\}
\end{equation}
and the following relations for all $\bi,\bj\in I^\al$ and all admissible $r$ and $s$:
\begin{equation}
e(\bi) e(\bj) = \de_{\bi,\bj} e(\bi),
\quad{\textstyle\sum_{\bi \in I^\alpha}} e(\bi) = 1;\label{R1}
\end{equation}
\begin{equation}\label{R2PsiY}
y_r e(\bi) = e(\bi) y_r;
\end{equation}
\begin{equation}
\psi_r e(\bi) = e(s_r\bi) \psi_r;\label{R2PsiE}
\end{equation}
\begin{equation}\label{R3Y}
y_r y_s = y_s y_r;
\end{equation}
\begin{equation}\label{R3YPsi}
y_r \psi_s = \psi_s y_r\qquad (r \neq s,s+1);
\end{equation}
\begin{equation}
(y_{r+1} \psi_r-\psi_r y_r) e(\bi) =
\left\{
\begin{array}{ll}
e(\bi) &\hbox{if $i_r=i_{r+1}$,}\\
0 &\hbox{if $i_r\neq i_{r+1}$;}
\end{array}
\right.
\label{R5}
\end{equation}
\begin{equation}
(\psi_r y_{r+1} -y_r\psi_r)e(\bi) 
= 
\left\{
\begin{array}{ll}
e(\bi) &\hbox{if $i_r=i_{r+1}$,}\\
0 &\hbox{if $i_r\neq i_{r+1}$;}
\end{array}
\right.
\label{R6}
\end{equation}
\begin{equation}
\psi_r^2e(\bi) = 
\left\{
\begin{array}{ll}
0 &\hbox{if $i_r=i_{r+1}$,}\\
e(\bi) & \hbox{if $a_{i_r i_{r+1}}=0$,}
\\
(y_r-y_{r+1})e(\bi)
 &\hbox{if $i_r\rightarrow i_{r+1}$,}
 \\
(y_{r+1}-y_{r})e(\bi)
 &\hbox{if $i_{r+1}\rightarrow i_{r}$;}
\end{array}
\right.
 \label{R4}
\end{equation}
\begin{equation} 
\psi_r \psi_s = \psi_s \psi_r\qquad (|r-s|>1);\label{R3Psi}
\end{equation}
\begin{equation}
\begin{split}
(\psi_{r+1}\psi_{r} \psi_{r+1}-\psi_{r} \psi_{r+1} \psi_{r}) e(\bi) 
=
\left\{\begin{array}{ll}
e(\bi)&\text{if $i_{r+2}=i_r\rightarrow i_{r+1}$,}\\
-e(\bi)&\text{if $i_{r+1}\rightarrow i_r=i_{r+2}$,}\\
0 &\text{otherwise.}
\end{array}\right.
\end{split}
\label{R7}
\end{equation}
The grading on $R_\al$ is defined by 
$$
\deg(e(\bi))=0,\quad \deg(y_re(\bi))=2,\quad\deg(\psi_r e(\bi))=-a_{i_ri_{r+1}}.
$$

\subsection{Basis Theorem}
For each element $w\in S_d$ fix a reduced expression $w=s_{i_1}\dots s_{i_m}$ and set 
$$
\psi_w:=\psi_{i_1}\dots \psi_{i_m}.
$$
In general, $\psi_w$ is not independent of the choice of reduced expression of $w$.

\begin{Theorem}\label{TBasis}{\cite[Theorem 2.5]{KL1}}%
{\bf (Basis Theorem)}
The elements 
\begin{equation}\label{EBasis}
\{\psi_w y_1^{m_1}\dots y_d^{m_d}e(\bi)\mid w\in S_d,\ m_1,\dots,m_d\in\Z_{\geq 0}, \ \bi\in I^\al\}
\end{equation}
form an $F$-basis for $R_\al$. 
\end{Theorem}

Denote by $P_\al$ the (commutative) subalgebra of $R_\al$ generated by $y_1,\dots,y_d$ and all $\{e(\bi)\mid \bi\in I^\al\}$. By the Basis Theorem, 
$$
\{y_1^{m_1}\dots y_d^{m_d}e(\bi)\mid m_1,\dots, m_d\in\Z_{\geq 0},\ \bi\in I^\al\}
$$
is a basis of $P_\al$. 

\subsection{Modules, weights, and characters}
If $V=\oplus_{k\in \Z} V[k]$ is a $\Z$-graded vector space, its {\em graded dimension} is 
$$
\gdim V:=\sum_{k\in \Z}(\dim V[k])q^k\in\Z[q,q^{-1}]. 
$$ 
Recall that $R_\al$ is a $\Z$-graded algebra. All $R_\al$-modules will be assumed {\em graded}, unless otherwise stated. We will work in the category 
$$
\mod{R_\al}=\{\text{finite dimensional graded $R_\al$-modules}\}.
$$
Since all $y_re(\bi)$ are positively graded, the elements $y_r$ act nilpotently on all modules $M\in\mod{R_\al}$. 

For every $\bi\in I^\al$ and any $M\in\mod{R_\al}$, the {\em $\bi$-weight space} of $M$ is
$
M_\bi:=e(\bi)M.
$
We have a decomposition of (graded) vector spaces
$$
M=\bigoplus_{\bi\in I^\al}M_\bi.
$$
We say that $\bi$ is a {\em weight of $M$} if $M_\bi\neq 0$, and refer to $I^\al$, as the set of {\em weights for $R_\al$}. Note by (\ref{R2PsiE}) that 
\begin{equation}\label{EAction}
\psi_r M_\bi\subseteq M_{s_r \bi}.
\end{equation}
Let $\Z[q,q^{-1}][I^\al]$ be the free $\Z[q,q^{-1}]$-module with basis $\{e^\bi\mid \bi\in I^\al\}$. The {\em formal character} of the module $M\in\mod{R_\al}$ is 
\begin{equation*}
\ch M:=\sum_{\bi\in I^\al}(\gdim M_\bi) e^{\bi}.
\end{equation*}

The formal character map $\ch: \mod{R_\al}\to \Z[q,q^{-1}][I^\al]$ factors through to give a $\Z[q,q^{-1}]$-linear map from the Grothendieck group
\begin{equation}\label{EChMap}
\ch: K(\mod{R_\al})\to \Z[q,q^{-1}][I^\al].
\end{equation}
The following result shows that the characters of the irreducible $R_\al$-modules are linearly independent. 

\begin{Theorem}\label{TFCh}{\rm \cite[Theorem~3.17]{KL1}}\,
The map (\ref{EChMap}) is injective. 
\end{Theorem}

\subsection{Weight graph} 
Let $1\leq r<d$ and $\bi\in I^\al$. We call $s_r$ an {\em admissible transposition} for $\bi$ if $a_{i_r i_{r+1}}=0$. 
By (\ref{R4}), if $\bi$ is a weight of $M\in\mod{R_\al}$ and $s_r$ is an admissible transposition for $\bi$, then $\gdim M_\bi=\gdim M_{s_r\bi}$. This explains our interest in the following combinatorial object. 

Define the {\em weight graph} $G_\al$ as the graph with the set of vertices $I^\al$, and with $\bi,\bj\in I^\al$ connected by an edge if and only if $\bj=s_r \bi$ for some admissible transposition $s_r$ for $\bi$. 
We want to describe the connected components of $G_\al$. 

Let $\bi\in I^\al$, and $a,b\in I$ be neighbors in $\Ga$. The \emph{$\{a,b\}$-sequence of
$\bi$} is the sequence of $a$'s and $b$'s obtained by ignoring all entries of $\bi$ different from 
$a$ and $b$.  For example, the $\{1,2\}$ sequence of $\bi=(1,2,2,3,4,1,2,1)$ is $(1,2,2,1,2,1)$.  Note that if $s_r$ is admissible transposition for $\bi$ then the $\{a,b\}$-sequence of $\bi$ is the same as the $\{a,b\}$-sequence of $s_r\bi$ for every pair  of neighbors $a,b\in I$. So the $\{a,b\}$-sequences are invariants of connected components of $G_\al$. It turns out that these invariants are enough to describe the components:

\begin{Proposition}\label{PComb}
Let $\bi,\bj\in I^\al$. Then $\bi$ and $\bj$ belong to the same connected component of $G_\al$ if and only if their $\{a,b\}$-sequences coincide for each pair of neighbors $a,b\in I$. 
\end{Proposition}
\begin{proof}
We prove the result by induction on $d=\height(\al)$.
Assume that $\bi=(i_1,\ldots, i_d)$ and $\bj=(j_1,\ldots, j_d)$ are elements of $I^\alpha$ so that
the $\{a,b\}$-sequences of $\bi$ and $\bj$ coincide for all pairs of neighbors $a,b\in I$.
If $d=1$ then $\bi=\bj$, and so $\bi$ and $\bj$ are in the same connected component of $I^\alpha$.
If $d>1$ let
$b = j_d$ and let $a$ be a neighbor of $b$.  Let $k$ be maximal such that $i_k=b$.  None of $i_{k+1},\ldots, i_d$ is equal to $a$.  Therefore $\bi$ is connected to 
$$\bi' = s_{d-1}\cdots s_{k+1}s_k\bi = (i_1,\ldots, i_{k-1}, i_{k+1},\ldots, i_d,b).$$  
Now $\bi'$ and $\bj$ are in the same connected component since, by inductive assumption, $(i_1,\ldots, i_{k-1}, i_{k+1},\ldots, i_{d})$ and $(j_1,\ldots, j_{d-1})$ are in the same connected component of $G_{\alpha - \alpha_b}$.
\end{proof}

\subsection{Configurations and standard tableaux}\label{SSConf}
We suggest  `geometric' objects called configurations to visualize connected components of $G_\al$. First, the {\em $\Gamma$-abacus} is $\Gamma\times\R_{\geq 0}$, imagined as the abacus with the runners going up on each vertex of $\Gamma$. We picture the $\Gamma$-abacus in $\R^3$ with the distance 
between neighboring runners always equal to $1$. For example, 
for $\Gamma =D_4$ and $\Gamma=A_\infty$ the abaci look like this:
$$
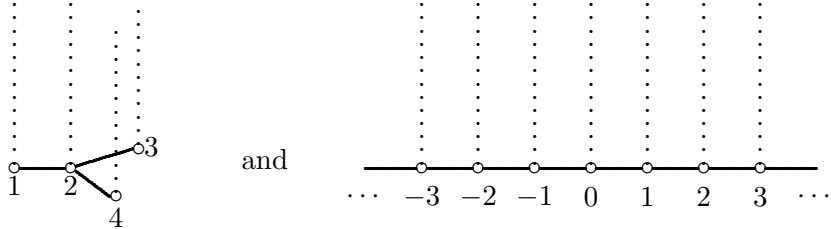
\beginpicture
\setcoordinatesystem units <.75cm,.75cm>         
\setplotarea x from -2 to 2, y from -1 to 3  
\put{$1$} at -1 -0.3 
\put{$2$} at 0 -0.3 
\put{$3$}[l] at 1.3 0.4 
\put{$4$}[t] at 0.8 -0.7 
\put{$\circ$} at -1 0
\plot -0.9 0  -0.1 0 / 
\put{$\circ$} at 0 0
\plot 0.1 0.05  1.1 0.35 / 
\put{$\circ$} at 1.2 0.35
\plot 0.1 -0.05  0.7 -0.5 / 
\put{$\circ$} at 0.8 -0.5
\setdots
\plot -1 0.1 -1 3 /
\plot 0 0.1 0 3 /
\plot 1.2 0.45  1.2 3 /
\plot 0.8 -0.4  0.8 2.6 /
\endpicture
\qquad \text{and}\qquad 
\beginpicture
\setcoordinatesystem units <.75cm,.75cm>         
\setplotarea x from -4 to 4, y from -1 to 3  
\put{$-3$} at -3 -0.5 
\put{$-2$} at -2 -0.5 
\put{$-1$} at -1 -0.5 
\put{$0$} at 0 -0.5 
\put{$1$} at 1 -0.5 
\put{$2$} at 2 -0.5 
\put{$3$} at 3 -0.5 
\put{$\cdots$} at 4 -0.5
\put{$\cdots$} at -4 -0.5
\plot -4 0  -3.1 0 / 
\put{$\circ$} at -3 0
\plot -2.9 0  -2.1 0 / 
\put{$\circ$} at -2 0
\plot -1.9 0  -1.1 0 / 
\put{$\circ$} at -1 0
\plot -0.9 0  -0.1 0 / 
\put{$\circ$} at 0 0
\plot 0.1 0  0.9 0 / 
\put{$\circ$} at 1 0
\plot 1.1 0  1.9 0 / 
\put{$\circ$} at 2 0
\plot 2.1 0  2.9 0 / 
\put{$\circ$} at 3 0
\plot 3.1 0  4 0 / 
\setdots
\plot -3 0.1 -3 3 /
\plot -2 0.1 -2 3 /
\plot -1 0.1 -1 3 /
\plot  0 0.1  0 3 /
\plot  1 0.1  1 3 /
\plot  2 0.1  2 3 /
\plot  3 0.1  3 3 /
\endpicture
$$

The `beads' of the abacus have shape depending on the runners. The bead on runner $i$ is `glued' out of isosceles  right triangles with hypotenuse of length 2 on the runner, and the $90^\circ$ vertex sticking towards the neighboring runner (and touching it). 
Examples of a bead on runner $1$ for type $A_3$, a bead on runner $i$ for type  $A_{\infty}$, a bead on runner $2$ for type $D_4$, and a bead on runner $1$ for type $A_1$ are:
$$
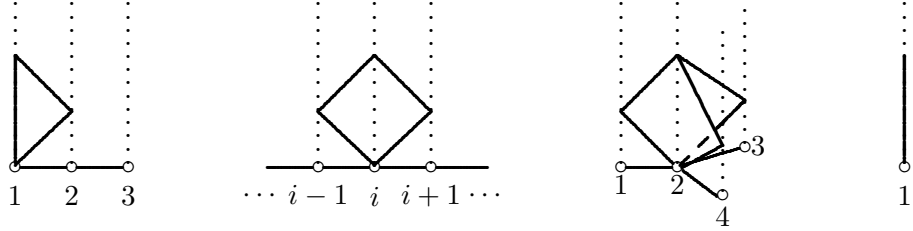
\beginpicture
\setcoordinatesystem units <.75cm,.75cm>         
\setplotarea x from .5 to 4, y from -1 to 3  
\put{$1$} at 1 -0.5 
\put{$2$} at 2 -0.5 
\put{$3$} at 3 -0.5 
%
\put{$\circ$} at 1 0
\plot 1.1 0  1.9 0 / 
\put{$\circ$} at 2 0
\plot 2.1 0  2.9 0 / 
\put{$\circ$} at 3 0
\plot 1 0.1  1 2 / 
\plot 1 2 2 1 /
\plot 2 1 1 0.05 /
\setdots
\plot  1 0.1  1 3 /
\plot  2 0.1  2 3 /
\plot  3 0.1  3 3 /
\endpicture
\qquad
\beginpicture
\setcoordinatesystem units <.75cm,.75cm>         
\setplotarea x from 0 to 3.5, y from -1 to 3  
\put{$i-1$} at 1 -0.5 
\put{$i$} at 2 -0.5 
\put{$i+1$} at 3 -0.5 
\put{$\cdots$} at 0 -0.5
\put{$\cdots$} at 4 -0.5
%
\plot 0.1 0  0.9 0 / 
\put{$\circ$} at 1 0
\plot 1.1 0  1.9 0 / 
\put{$\circ$} at 2 0
\plot 2.1 0  2.9 0 / 
\put{$\circ$} at 3 0
\plot 3.1 0  4 0 / 
\plot 2 2 3 1 /
\plot 2 2 1 1 /
\plot 3 1 2 0.05 /
\plot 1 1 2 0.05 /
\setdots
\plot  1 0.1  1 3 /
\plot  2 0.1  2 3 /
\plot  3 0.1  3 3 /
\endpicture
\qquad
\beginpicture
\setcoordinatesystem units <.75cm,.75cm>         
\setplotarea x from -2 to 2, y from -1 to 3  
\put{$1$} at -1 -0.3 
\put{$2$} at 0 -0.3 
\put{$3$}[l] at 1.3 0.4 
\put{$4$}[t] at 0.8 -0.7 
\put{$\circ$} at -1 0
\plot -0.9 0  -0.1 0 / 
\put{$\circ$} at 0 0
\plot 0.1 0.05  1.1 0.35 / 
\put{$\circ$} at 1.2 0.35
\plot 0.1 -0.05  0.7 -0.5 / 
\put{$\circ$} at 0.8 -0.5
\plot 0 2 1.2 1.2 /
\plot 1.2 1.2 0.7 0.7 /
\plot 0 2 0.8 0.4 /
\plot 0.8 0.4  0.08 0 /
\plot 0 2  -1 1 /
\plot  -1 1 -0.08 0.08 /
\setdots
\plot -1 0.1 -1 3 /
\plot 0 0.1 0 3 /
\plot 1.2 0.45  1.2 3 /
\plot 0.8 -0.4  0.8 2.6 /
\setdashes
\plot 1.2 1.2 0 0 /
\endpicture
\qquad
\beginpicture
\setcoordinatesystem units <.75cm,.75cm>         
\setplotarea x from 0 to 2, y from -1 to 3  
\put{$1$} at 1 -0.5 
%
\put{$\circ$} at 1 0
\plot 1 0.1  1 2 / 
\setdots
\plot  1 0.1  1 3 /
\endpicture
$$
Note that if $i$ has no neighbors, the shape of the bead is interpreted as just a segment of length $2$ (`hypotenuse without triangles').

Recall that $\al=\sum_{i\in I}m_i\al_i$ is a fixed element of $Q_+$ of height $d$. 
A {\em configuration} of type $\al$ is obtained by placing $d$ beads on the runners of the $\Gamma$-abacus, letting each bead slide down the runner as far as gravity takes it, 
so that there are a total of $m_i$ beads on runner $i$ for each $i\in I$. 
We note that configurations are essentially the same as heaps defined by Viennot \cite{V}, see also Stembridge \cite{S1,S2}. 

Let $\la$ be a configuration. A {\em tableau} of shape $\la$ or a {\em $\la$-tableau} is a bijection
$$
T:\{1,2,\dots,d\}\to\{\text{beads of $\la$}\}.
$$
A bead $B$ of $\lambda$ is {\em removable} if it can be lifted off its runner without interfering with other beads. If $B$ is on runner $i$, this is equivalent to the requirement that there are no beads on neighboring runners which are above $B$ in $\lambda$.  
A $\lambda$-tableaux is called {\em standard} if for each $k$, 
the bead $T(k)$ is above the bead $T(m)$ whenever $m<k$ and $T(m)$ is on a neighboring runner. 
Equivalently, $T$ is standard, if and only if $T(k)$ is a removable bead for the configuration $\lambda\setminus \{T(k+1),\dots,T(d)\}$ for all $1\leq k\leq d$.

Let $\bi=(i_1,\dots,i_d)\in I^\al$. Place a bead on the runner $i_1$, then place a bead on the runner $i_2$, and so on, finally placing the last bead on the runner $i_d$. This procedure produces the {\em configuration of $\bi$}, written $\sh(\bi)=\sh_\Gamma(\bi)$, and the standard tableaux $T^\bi$ of the corresponding shape. For example:
$$\text{In type $A_\infty$,}\ T^{(0,-2,2)}=
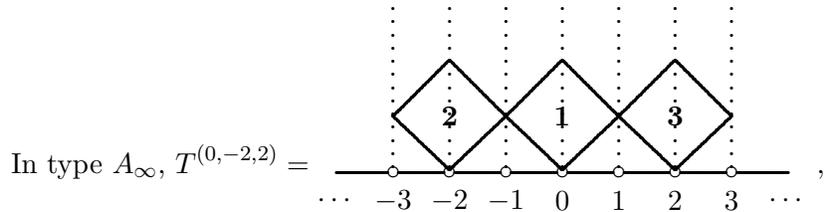
\beginpicture
\setcoordinatesystem units <.75cm,.75cm>         
\setplotarea x from -4 to 4, y from -1 to 3  
\put{$-3$} at -3 -0.5 
\put{$-2$} at -2 -0.5 
\put{$-1$} at -1 -0.5 
\put{$0$} at 0 -0.5 
\put{$1$} at 1 -0.5 
\put{$2$} at 2 -0.5 
\put{$3$} at 3 -0.5 
\put{$\cdots$} at 4 -0.5
\put{$\cdots$} at -4 -0.5
\put{${\mathbf 2}$} at -2 1
\put{${\mathbf 1}$} at 0 1
\put{${\mathbf 3}$} at 2 1
\plot -4 0  -3.1 0 / 
\put{$\circ$} at -3 0
\plot -2.9 0  -2.1 0 / 
\put{$\circ$} at -2 0
\plot -1.9 0  -1.1 0 / 
\put{$\circ$} at -1 0
\plot -0.9 0  -0.1 0 / 
\put{$\circ$} at 0 0
\plot 0.1 0  0.9 0 / 
\put{$\circ$} at 1 0
\plot 1.1 0  1.9 0 / 
\put{$\circ$} at 2 0
\plot 2.1 0  2.9 0 / 
\put{$\circ$} at 3 0
\plot 3.1 0  4 0 / 

\plot 2 2 3 1 /
\plot 2 2 1 1 /
\plot 3 1 2 0.05 /
\plot 1 1 2 0.05 /

\plot 0 2 1 1 /
\plot 0 2 -1 1 /
\plot 1 1 0 0.05 /
\plot -1 1 0 0.05 /

\plot -2 2 -1 1 /
\plot -2 2 -3 1 /
\plot -1 1 -2 0.05 /
\plot -3 1 -2 0.05 /

\setdots
\plot -3 0.1 -3 3 /
\plot -2 0.1 -2 3 /
\plot -1 0.1 -1 3 /
\plot  0 0.1  0 3 /
\plot  1 0.1  1 3 /
\plot  2 0.1  2 3 /
\plot  3 0.1  3 3 /
\endpicture
\ ,
$$
$$
\text{in type $A_\infty$,}\ T^{(0,1,-1)}=
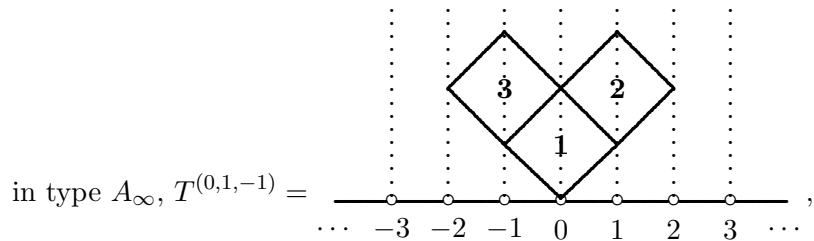
\beginpicture
\setcoordinatesystem units <.75cm,.75cm>         
\setplotarea x from -4 to 4, y from -1 to 3.5  
\put{$-3$} at -3 -0.5 
\put{$-2$} at -2 -0.5 
\put{$-1$} at -1 -0.5 
\put{$0$} at 0 -0.5 
\put{$1$} at 1 -0.5 
\put{$2$} at 2 -0.5 
\put{$3$} at 3 -0.5 
\put{$\cdots$} at 4 -0.5
\put{$\cdots$} at -4 -0.5
\put{${\mathbf 1}$} at 0 1
\put{${\mathbf 2}$} at 1 2
\put{${\mathbf 3}$} at -1 2
\plot -4 0  -3.1 0 / 
\put{$\circ$} at -3 0
\plot -2.9 0  -2.1 0 / 
\put{$\circ$} at -2 0
\plot -1.9 0  -1.1 0 / 
\put{$\circ$} at -1 0
\plot -0.9 0  -0.1 0 / 
\put{$\circ$} at 0 0
\plot 0.1 0  0.9 0 / 
\put{$\circ$} at 1 0
\plot 1.1 0  1.9 0 / 
\put{$\circ$} at 2 0
\plot 2.1 0  2.9 0 / 
\put{$\circ$} at 3 0
\plot 3.1 0  4 0 / 

\plot 2 2  1  3 /
\plot 0 2 1 3 /
\plot 1 1 2 2 /

\plot 0 2 1 1 /
\plot 0 2 -1 1 /
\plot 1 1 0 0.05 /
\plot -1 1 0 0.05 /

\plot -2 2 -1 1 /
\plot -2 2  -1 3 /
\plot -1 3 0 2 /

\setdots
\plot -3 0.1 -3 3.5 /
\plot -2 0.1 -2 3.5 /
\plot -1 0.1 -1 3.5 /
\plot  0 0.1  0 3.5 /
\plot  1 0.1  1 3.5 /
\plot  2 0.1  2 3.5 /
\plot  3 0.1  3 3.5 /
\endpicture,
$$
$$
\text{in type $D_4$,}\ \sh(2,1,3,4,2)=
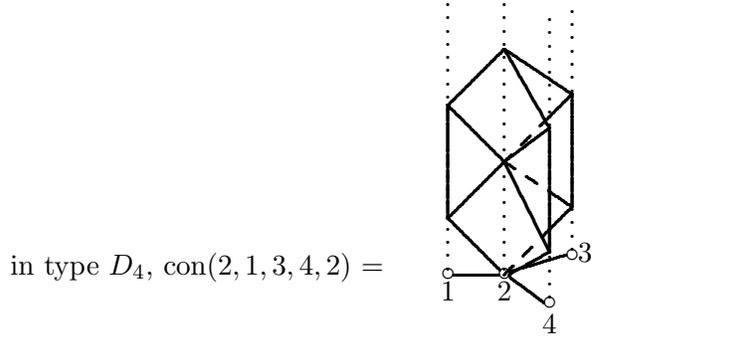
\beginpicture
\setcoordinatesystem units <.75cm,.75cm>         
\setplotarea x from -2 to 4, y from -1 to 5  
\put{$1$} at -1 -0.3 
\put{$2$} at 0 -0.3 
\put{$3$}[l] at 1.3 0.4 
\put{$4$}[t] at 0.8 -0.7 
\put{$\circ$} at -1 0
\plot -0.9 0  -0.1 0 / 
\put{$\circ$} at 0 0
\plot 0.1 0.05  1.1 0.35 / 
\put{$\circ$} at 1.2 0.35
\plot 0.1 -0.05  0.7 -0.5 / 
\put{$\circ$} at 0.8 -0.5
\plot 0 2 0.8 0.4 /
\plot 0.8 0.4  0.08 0 /
\plot 0 2  -1 1 /
\plot  -1 1 -0.08 0.08 /
\plot -1 3 0 2 /
\plot -1 1 -1 3 /
\plot 1.2 3.2 0.79 2.79 /


\plot 1.2 1.2 1.2 3.2 /

\plot 0.8 2.6 0 2 /
\plot 0.8 0.4 0.8 2.55 /
\plot 0 4 -1 3 /
\plot 0 4 1.2 3.2 /
\plot 0 4 0.8 2.6 /

\plot .8 1.47 1.2 1.2 /
\plot 0.8 0.8 1.2 1.2 /

\setdots
\plot -1 0.1  -1 5 /
\plot 0 0.1 0 5 /
\plot 1.2 0.45   1.2 5 /
\plot 0.8 -0.4  0.8 4.6 /

\setdashes

\plot 0 2 .8 1.47 /

\plot 0 0  0.8 0.8 /

\plot 0.79 2.79 0 2 /

\endpicture
\ .
$$
The reader might note that in type $A_\infty$ configurations are closely related to the `Russian' notation for Young diagrams, cf. \cite{VK,Ok}. 

For any $\lambda$-tableau $T$ we denote by $\bi^T$ the element 
$$\bi^T=(i^T_1,\dots,i^T_d)\in I^\al,$$ 
where $i^T_k$ is the label of the runner occupied by the bead $T(k)$ ($1\leq k\leq d$). 
Now note that the maps $T\mapsto \bi^T$ and $\bi\mapsto T^\bi$ are mutually inverse bijections between the set $\mathcal{T}(\la)$ of the standard $\lambda$-tableaux and the set of weights $\bi\in I^\al$ with $\sh(\bi)=\lambda$. 
Now we can interpret Proposition~\ref{PComb} as the following statement:

\begin{Proposition}\label{PShape}
Two weights $\bi,\bj\in I^\al$ are in the same connected component of $G_\al$ if and only if $\sh(\bi)=\sh(\bj)$. Moreover, the maps $T\mapsto \bi^T$ and $\bi\mapsto T^\bi$ are mutually inverse bijections between the set of the standard $\lambda$-tableaux and the set of all weights $\bi\in I^\al$ with $\sh(\bi)=\lambda$. 
\end{Proposition}

\section{Homogeneous representations}
We continue working with a fixed graph $\Gamma$ and a fixed $\al=\sum_{i\in I}m_i\al_i\in Q_+$ of height $d$. 
A module $M\in \mod{R_\al}$ is called \emph{homogeneous} if it is concentrated in one degree,
i.e. $M=M[k]$ for some $k\in \Z$. 
The homogeneous irreducible modules are especially easy to understand. They are labeled by `skew shapes', and their formal characters are `sums of standard tableaux' of that shape. 

\subsection{Calibrated representations} 
First, we consider a seemingly different class of modules. A module $M\in\mod{R_\al}$ is called {\em calibrated} if $y_1,\dots,y_d$ act as zero on $M$. 
Other authors might use different terminology here, for example {\em Gelfand-Zetlin} \cite{Ch,OV}, {\em completely splittable} \cite{KCS,Kbook, R}, {\em skew} \cite{Mo}, {\em seminormal} \cite{Ma}, etc.
Our goal is to classify irreducible calibrated modules following the approach of \cite{KrR}. 

\begin{Proposition}\label{PCal1}
Let $M\in\mod{R_\al}$ be an irreducible calibrated module, and $\bi$ be a weight of $M$. Then: 
\begin{enumerate}
\item[{\rm (i)}] there is no $r$ with $i_r= i_{r+1}$;
\item[{\rm (ii)}]  there is no $r$ such that $i_r,i_{r+1}$ are neighbors and $i_{r+2}=i_r$; 
\item[{\rm (iii)}] $\dim M_\bi=1$;
\item[{\rm (iv)}] the weights of $M$ form one connected component of $G_\al$.
\end{enumerate}
\end{Proposition}
\begin{proof}
(i) Assume $i_r = i_{r+1}$ and let $v\in M_\bi$ be nonzero.  Since $M$ is calibrated, $y_r$ and $y_{r+1}$ act as $0$, and \eqref{R5} leads to a contradiction:
$$0 = (y_{r+1}\psi_r - \psi_r y_r)e(\bi)v= e(\bi)v=v.$$

(ii)  Assume $(i_r,i_{r+1},i_{r+2}) = (a,b,a)$, $a$ and $b$ are neighbors, and 
$v\in M_\bi$ is nonzero.  By \eqref{R2PsiE}, $\psi_{r+1}v\in M_{s_{r+1}\bi}$ and 
$\psi_r v\in M_{s_r\bi}$. So, by (i), we have $\psi_{r+1}v = 0$ and $\psi_rv=0$.  
Using \eqref{R7}, we get a contradiction: 
$$0 = (\psi_{r+1}\psi_r\psi_{r+1} - \psi_r \psi_{r+1}\psi_r)v =\pm v.
$$

(iii) Assume for a contradiction that $v,w$ are two linearly independent elements of $M_\bi$. As $M$ is irreducible and calibrated, we may assume (up to rescaling) that $v=\psi_{r_1}\psi_{r_2}\dots\psi_{r_k}w$ and that $k$ is minimal possible. It follows from (\ref{R2PsiE}) and (i) that $s_{r_1}s_{r_2}\dots s_{r_k}=1$ in $S_d$. So we can use braid relations to rewrite 
$$
s_{r_1}\dots s_{r_k}=s_{t_1}\dots s_{t_{m-2}}s_{t}s_{t}s_{t_{m+1}}\dots s_{t_k}.
$$ 
By (ii) and (\ref{R7}), $\psi_r$'s acting on $M$ also satisfy braid relations, so we can rewrite, using also (\ref{R4}), 
\begin{align*}
\psi_{r_1}\dots \psi_{r_k}w
&=\psi_{t_1}\dots \psi_{t_{m-2}}\psi_{t}\psi_{t}\psi_{t_{m+1}}\dots \psi_{t_k}w
\\
&= c\psi_{t_1}\dots \psi_{t_{m-2}}\psi_{t_{m+1}}\dots \psi_{t_k}w
\end{align*}
for some constant $c$, which must be non-zero, and hence $c=1$. 
This contradicts the minimality of $k$. 

(iv) If $\bi$ is a weight of $M$, and $s_r$ is an admissible transposition for $\bi$, then $s_r\bi$ is also a weight of $M$, thanks to (\ref{R2PsiE}) and (\ref{R4}). So all weights in the connected component of $\bi$ in $G_\al$ appear in $M$. To see that there are no other weights, it suffices to show that if $\bj$ and  $s_r\bj$ are weights of $M$ then $s_r$ is an admissible transposition for $\bj$. 

So let $v\in M_\bj$, $w\in M_{s_r\bj}$ be non-zero vectors. After rescaling, we may assume that $w=\psi_{r_1}\dots \psi_{r_k}v$, and let $k$ be  minimal possible. By (\ref{R2PsiE}) and (i), $s_{r_1}\dots s_{r_k}=s_r$ in $S_d$. As in the proof of (iii), we deduce from the minimality of $k$ that $k=1$ and $r_1=r$, i.e. $w=\psi_r v$. Similarly, we can write $cv=\psi_rw$ for a non-zero constant $c$. So $\psi_r^2v\neq 0$. In view of (\ref{R4}), $j_r$ and $j_{r+1}$ are not neighbors, whence $s_r$ is an admissible transposition for $\bj$. 
\end{proof}

\begin{Corollary}\label{CGZ}
Let $M\in\mod{R_\al}$ be an irreducible module. Then $M$ is homogeneous if and only if $M$ is calibrated. 
\end{Corollary}
\begin{proof}
If $M$ is homogeneous, then $y_1,\dots,y_d$ act on $M$ as zero since they have positive degrees. Conversely, if $M$ is calibrated, it follows from Proposition~\ref{PCal1} that $M$ is a span of some $\psi_{r_1}\dots\psi_{r_k}v$ where $v\in M_\bi$ for some $\bi$, and $s_{r_m}$ is an admissible transposition for $s_{r_{m+1}}\dots s_{r_k}\bi$,  for all $m=1,\dots,k$. It follows that the degree of each $\psi_{r_1}\dots\psi_{r_k}v$ is the same as the degree of $v$, so $M$ is homogeneous. 
\end{proof} 

\subsection{Construction of homogeneous modules}
We now give an explicit construction of the homogeneous representations, which can be thought of as a generalization of Young's seminormal form \cite{Ch} from type $A_\infty$ quiver to an arbitrary quiver without loops and multiple edges. 

Let $C$ be a connected component of $G_\al$. We say that $C$ is {\em homogeneous} if for each $\bi\in C$ the following condition holds:
\begin{equation}\label{ENC}
\begin{split}
\text{if $i_r=i_s$ for some $r<s$ then there exist $t,u$ with}
\\  
\text{$r<t<u<s$ such that 
$a_{i_ri_t}=a_{i_r,i_u}=-1$.}
\end{split}
\end{equation}

\begin{Lemma}\label{equivconds}
Let $C$ be a connected component of $G_\al$. 
\begin{enumerate}
\item[{\rm (i)}] $C$ is homogeneous if and only if the condition (\ref{ENC}) holds for {\em some} $\bi\in C$.
\item[{\rm (ii)}]  $C$ is homogeneous if and only if 
the conditions (i) and (ii) of Proposition~\ref{PCal1} hold for each $\bi\in C$. 
\end{enumerate}
\end{Lemma}
\begin{proof}
(i) Condition (3.1) is a condition on the $\{a,b\}$-sequences of $\bi$ which requires that
$$\bi = \cdots a \cdots a\cdots
\quad\hbox{only if}\quad
\bi = \cdots a \cdots b \cdots c\cdots a\cdots
$$
with $b$ and $c$ distinct neighbors of $a$.  If this condition holds for one $\bi\in C$ then,
by Proposition~\ref{PComb}, it holds for all $\bi\in C$.

(ii) `$\Rightarrow$':  If Proposition 3.1 (i) or (ii) is violated then there exists $\bi\in C$ with
$$\bi = \cdots aa \cdots\quad\hbox{or}\quad
\bi = \cdots aba\cdots,
$$
with $b$ a neighbor of $a$.  In either case $\bi$ violates the condition in (3.1).

`$\Leftarrow$': If condition (3.1) is violated then there exists $\bi \in C$ such that $\bi$ looks like
$$\hbox{Case 1:}\quad \bi = \cdots a \cdots a\cdots,$$
with $a=i_r=i_s$ and no neighbors of $a$ in between, or
$$\hbox{Case 2:}\quad \bi = \cdots a \cdots b\cdots a\cdots,$$
with $a = i_r=i_s$, $b=i_t$ a neighbor of $a$ and no other neighbors of $a$ in between $i_r$ and $i_s$.
In Case 1, $\bi$ is connected to
$$\bj = s_{i_s-1}\cdots s_{i_r+1}s_{i_r}\bi = \cdots aa\cdots,$$
which violates Proposition 3.1(i).  In Case 2, $\bi$ is connected to
$$\bj = (s_{i_t-1}\cdots s_{i_r+1}s_{i_r})(s_{i_t+1}\cdots s_{i_s-2}s_{i_s-1})\bi = \cdots aba\cdots ,$$
which violates Propositions 3.1 (ii).
\end{proof}

%

\begin{Theorem}\label{THomog}
Let $C$ be a homogeneous connected component of $G_\al$, and let us consider a vector space $S(C)$ with a homogeneous basis $\{v_\bi\mid \bi\in C\}$ labeled by the elements of $C$. 
The formulas
\begin{align*}
e(\bj)v_\bi&=\de_{\bi,\bj}v_\bi \qquad (\bj\in I^\al,\ \bi\in C),\\
y_r v_\bi&=0\qquad (1\leq r\leq d,\ \bi\in C),\\
\psi_rv_{\bi}&=
\left\{
\begin{array}{ll}
v_{s_r\bi} &\hbox{if $s_r\bi\in C$,}\\
0 &\hbox{otherwise;}
\end{array}
\right.
\quad(1\leq r<d,\ \bi\in C)
\end{align*}
define an action of $R_\al$ on $S(C)$, under which $S(C)$ is a homogeneous irreducible $R_\al$-module. Moreover, $S(C)\not\cong S(C')$ if $C\neq C'$, and every homogeneous irreducible $R_\al$-module is isomorphic to one of the modules $S(C)$.
\end{Theorem}
\begin{proof}
It is straightforward to verify that the formulas above define operators which satisfy the defining relations of $R_\al$, and so $S(C)$ is a well defined $R_\al$-module. It is also clear that it is concentrated in one degree, i.e. is homogeneous. The irreducibility of $S(C)$ follows from the definition of $C$ as a connected component of $G_\al$. If $C\neq C'$ then of course $S(C)$ is not isomorphic to $S(C')$ since they have different weights. Finally,   if $S$ is an irreducible homogeneous $R_\al$-module then by Corollary~\ref{CGZ} and Proposition~\ref{PCal1} the formal character of $S$ equals $\ch S(C)$ for some homogeneous connected component $C$, and so $S\cong S(C)$ thanks to Theorem~\ref{TFCh}.  
\end{proof}

\subsection{Skew shapes}\label{SSSS}
By Theorem~\ref{THomog}, the homogeneous connected components correspond to the homogeneous representations of $R_\al$. The homogeneous connected components are characterized by the properties (i) and (ii) from Proposition~\ref{PCal1}. The corresponding configurations can be characterized as follows:

\begin{Definition}\label{skshapes}
{\rm 
A configuration $\lambda$ is called {\em skew shape} if whenever $B_1$ and $B_2$ are two beads of $\lambda$ on the same runner then there are at least two beads on different  neighboring runners separating $B_1$ from $B_2$.  
}
\end{Definition}

For example, in type $A_\infty$, 

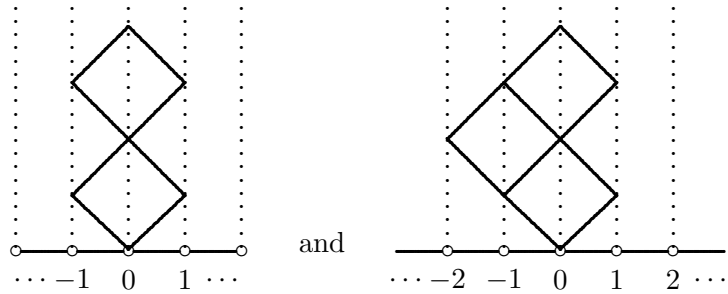
$$\beginpicture
\setcoordinatesystem units <.75cm,.75cm>         
\setplotarea x from -2 to 2.5, y from -1 to 4.5  
\put{$-1$} at -1 -0.5 
\put{$0$} at 0 -0.5 
\put{$1$} at 1 -0.5 
\put{$\cdots$} at 1.7 -0.5
\put{$\cdots$} at -1.7 -0.5
\put{$\circ$} at -2 0
\plot -1.9 0  -1.1 0 / 
\put{$\circ$} at -1 0
\plot -0.9 0  -0.1 0 / 
\put{$\circ$} at 0 0
\plot 0.1 0  0.9 0 / 
\put{$\circ$} at 1 0
\plot 1.1 0  1.9 0 / 
\put{$\circ$} at 2 0

\plot 0 2  1  3 /
\plot 0 4 1 3 /
\plot 0 4 -1 3 /

\plot 0 2 1 1 /
\plot 0 2 -1 1 /
\plot 1 1 0 0.05 /
\plot -1 1 0 0.05 /

\plot -1 3 0 2 /

\setdots
\plot -2 0.1 -2 4.5 /
\plot -1 0.1 -1 4.5 /
\plot  0 0.1  0 4.5 /
\plot  1 0.1  1 4.5 /
\plot  2 0.1  2 4.5 /
\endpicture
\quad 
\text{and}
\quad
\beginpicture
\setcoordinatesystem units <.75cm,.75cm>         
\setplotarea x from -3.3 to 3.1, y from -1 to 4.5  
\put{$-2$} at -2 -0.5 
\put{$-1$} at -1 -0.5 
\put{$0$} at 0 -0.5 
\put{$1$} at 1 -0.5 
\put{$2$} at 2 -0.5 
\put{$\cdots$} at 2.7 -0.5
\put{$\cdots$} at -2.7 -0.5
\plot -2.9 0  -2.1 0 / 
\put{$\circ$} at -2 0
\plot -1.9 0  -1.1 0 / 
\put{$\circ$} at -1 0
\plot -0.9 0  -0.1 0 / 
\put{$\circ$} at 0 0
\plot 0.1 0  0.9 0 / 
\put{$\circ$} at 1 0
\plot 1.1 0  1.9 0 / 
\put{$\circ$} at 2 0
\plot 2.1 0  2.9 0 / 

\plot 0 2  1  3 /
\plot 0 4 1 3 /
\plot 0 4 -1 3 /

\plot 0 2 1 1 /
\plot 0 2 -1 1 /
\plot 1 1 0 0.05 /
\plot -1 1 0 0.05 /

\plot -2 2 -1 1 /
\plot -2 2  -1 3 /
\plot -1 3 0 2 /

\setdots
\plot -2 0.1 -2 4.5 /
\plot -1 0.1 -1 4.5 /
\plot  0 0.1  0 4.5 /
\plot  1 0.1  1 4.5 /
\plot  2 0.1  2 4.5 /
\endpicture
$$

are not skew shapes, while

$$
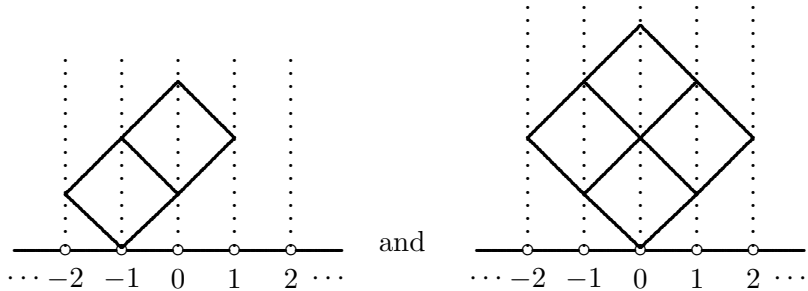
\beginpicture
\setcoordinatesystem units <.75cm,.75cm>         
\setplotarea x from -4 to 3, y from -1 to 3.5  
\put{$-2$} at -2 -0.5 
\put{$-1$} at -1 -0.5 
\put{$0$} at 0 -0.5 
\put{$1$} at 1 -0.5 
\put{$2$} at 2 -0.5 
\put{$\cdots$} at 2.7 -0.5
\put{$\cdots$} at -2.7 -0.5
\plot -2.9 0  -2.1 0 / 
\put{$\circ$} at -2 0
\plot -1.9 0  -1.1 0 / 
\put{$\circ$} at -1 0
\plot -0.9 0  -0.1 0 / 
\put{$\circ$} at 0 0
\plot 0.1 0  0.9 0 / 
\put{$\circ$} at 1 0
\plot 1.1 0  1.9 0 / 
\put{$\circ$} at 2 0
\plot 2.1 0  2.9 0 / 


\plot 0 3 1 2 /
\plot 0 3 -1 2 /
\plot 1 2 0 1 /
\plot -1 2 0 1 /

\plot 0 1 -1 0.05 /
\plot -1 0.05 -2 1 /
\plot -2 1 -1 2 /

\setdots
\plot -2 0.1 -2 3.5 /
\plot -1 0.1 -1 3.5 /
\plot  0 0.1  0 3.5 /
\plot  1 0.1  1 3.5 /
\plot  2 0.1  2 3.5 /
\endpicture
\quad
\text{and}
\quad
\beginpicture
\setcoordinatesystem units <.75cm,.75cm>         
\setplotarea x from -3.3 to 3.1, y from -1 to 4.5  
\put{$-2$} at -2 -0.5 
\put{$-1$} at -1 -0.5 
\put{$0$} at 0 -0.5 
\put{$1$} at 1 -0.5 
\put{$2$} at 2 -0.5 
\put{$\cdots$} at 2.7 -0.5
\put{$\cdots$} at -2.7 -0.5
\plot -2.9 0  -2.1 0 / 
\put{$\circ$} at -2 0
\plot -1.9 0  -1.1 0 / 
\put{$\circ$} at -1 0
\plot -0.9 0  -0.1 0 / 
\put{$\circ$} at 0 0
\plot 0.1 0  0.9 0 / 
\put{$\circ$} at 1 0
\plot 1.1 0  1.9 0 / 
\put{$\circ$} at 2 0
\plot 2.1 0  2.9 0 / 

\plot 0 2  1  3 /
\plot 0 4 1 3 /
\plot 0 4 -1 3 /

\plot 0 2 1 1 /
\plot 0 2 -1 1 /
\plot 1 1 0 0.05 /
\plot -1 1 0 0.05 /

\plot -2 2 -1 1 /
\plot -2 2  -1 3 /
\plot -1 3 0 2 /

\plot 1 3 2 2 /
\plot 2 2 1 1 /

\setdots
\plot -2 0.1 -2 4.5 /
\plot -1 0.1 -1 4.5 /
\plot  0 0.1  0 4.5 /
\plot  1 0.1  1 4.5 /
\plot  2 0.1  2 4.5 /
\endpicture
$$
are skew shapes. Note that, up to a horizontal shift, skew shapes in type $A_\infty$ are obtained by considering all usual skew shapes in the Russian notation and allowing all beads to slide down as far as gravity will take them.

If $\lambda$ is a configuration, then $S_d$ acts on the set of $\lambda$-tableaux by permutations of  $\{1,2,\dots,d\}$.  
Theorem~\ref{THomog} can now be restated as follows: 

\begin{Theorem}\label{tabconst}
Let $\lambda$ be a skew shape, and $\mathcal{T}(\lambda)$ be the set of all standard $\lambda$-tableaux. Consider a vector space $S(\lambda)$ with a homogeneous basis $\{v_T\mid T\in \mathcal{T}(\lambda)\}$. 
The formulas
\begin{align*}
e(\bj)v_T&=\de_{\bi^T\bj}v_T \qquad (\bj\in I^\al,\ T\in \mathcal{T}(\lambda)),\\
y_r v_T&=0\qquad (1\leq r\leq d,\ T\in \mathcal{T}(\lambda)),\\
\psi_rv_T&=
\left\{
\begin{array}{ll}
v_{s_rT} &\hbox{if $s_rT$ is standard,}\\
0 &\hbox{otherwise;}
\end{array}
\right.
\quad(1\leq r<d,\ T\in \mathcal{T}(\lambda))
\end{align*}
define an action of $R_\al$ on $S(\lambda)$, under which $S(\lambda)$ is a homogeneous irreducible $R_\al$-module. Moreover, $S(\lambda)\not\cong S(\lambda')$ if $\lambda\neq \lambda'$ and every homogeneous irreducible $R_\al$-module is isomorphic to one of the modules $S(\lambda)$.
\end{Theorem}

\subsection{Characters and the Littlewood-Richardson rule}

Let $\lambda$ be a skew shape and let $S(\lambda)$ be the corresponding
irreducible homogeneous $R_\alpha$-module constructed in Theorem \ref{tabconst}. 
Recall the maps $\bi\mapsto T^\bi$ and $T\mapsto \bi^T$ from \S\ref{SSConf}. 
Since $v_T$ is in the $\bi^T$-weight space, and this weight space is one dimensional,
the formal character of $R^\lambda_\alpha$ 
is
\begin{equation}\label{skewschur}
\ch S(\lambda) = \sum_{T\in \mathcal{T}(\lambda)} e^{i^T},
\end{equation}
where the sum is over all standard tableaux $T$ of shape $\lambda$. 

Let $\beta,\gamma\in Q_+$ such that $\beta+\gamma = \alpha$.  The product
$R_\beta\otimes R_\gamma$ is naturally a subalgebra of $R_\alpha$, cf. \cite[\S2.6]{KL1}.  If $M$ is a 
homogenous $R_\al$-module then its restriction to $R_\beta\otimes R_\gamma$ is homogenous.
It follows from \eqref{skewschur} that
\begin{equation}\label{LRrule}\Res^{R_\alpha}_{R_\beta\otimes R_\gamma}S(\lambda) 
= \sum_{\mu\subseteq \lambda} S(\mu)\otimes S(\lambda/\mu),
\end{equation}
where the sum is over all configurations $\mu$ of type $\beta$ which are obtained by consecutive removals of removable beads from
$\lambda$, 
and 
$\lambda/\mu$ is the configuration determined by the beads of $\lambda$ that are 
not in $\mu$.  The formula \eqref{LRrule} is a generalization of the skew Schur function formula from \cite[(5.10)]{Mac}:
$$
s_{\lambda/\mu}(x,y) = \sum_{\lambda \supseteq \nu\supseteq \mu}
s_{\lambda/\nu}(x)s_{\nu/\mu}(y)
$$

\subsection{Minuscule elements and hook formula} Finally, we explain a connection
between skew shapes and the fully commutative elements in Coxeter groups studied by  Stembridge \cite{S2} and Fan \cite{F}. A special class of fully commutative elements called dominant minuscule elements will allow us to select straight shapes from the class of skew shapes. 

Using notation of \cite{Kac}, let $\Phi_+$ be 
the set of positive roots, $<$  the dominance order, $P_+$ the set of {\em dominant weights}, 
and $W$ be the Weyl group with simple reflections $r_i$ for $\bi\in I$, so that $W$ is the Coxeter group 
with Coxeter graph $\Gamma$. 

An element $w\in W$ is {\em fully commutative} if for every pair of 
non-commuting generators $r_i$ and $r_j$ there is no reduced expression for $w$  
containing a subword of the form $r_i r_j r_i$.
An element $w\in W$ is {\em dominant minuscule} if there is $\La\in P_+$ and a reduced expression 
$w = r_{i_1} \dots r_{i_d}$ such that 
$$
r_{i_k}r_{i_{k+1}}\dots r_{i_d}\La=\La-\al_{i_k}-\al_{i_{k+1}}-\dots-\al_{i_d}\qquad(1\leq k\leq d).
$$

Using the terminology of \S\ref{SSSS}, let $\la$ be a skew shape and $\mathcal{T}(\la)$ 
the set of standard $\la$-tableaux. If $T\in \mathcal{T}(\la)$ and $\bi^T=(i_1,\dots,i_d)$, set 
\begin{equation}\label{ESt}
w^\la:=r_{i_d}r_{i_{d-1}}\dots r_{i_1}\in W.
\end{equation}


In view of Lemma \ref{equivconds} and Definition \ref{skshapes}, skew shapes and standard tableaux can now be interpreted as follows. 

\begin{Proposition}\label{PSt}
The element $w^\la$ depends only on $\la$ and does not depend on $T\in\mathcal{T}(\la)$. 
Moreover: 
\begin{enumerate}
\item[{\rm (i)}] the right hand side of (\ref{ESt}) is a reduced decomposition of $w^\la$;
\item[{\rm (ii)}] $\la\mapsto w^\la$ is a bijection between the skew shapes with $d$ boxes and the fully commutative elements of $W$ of length $d$;
\item[{\rm (iii)}] for a fixed skew shape $\la$, the assignment (\ref{ESt}) is a bijection between the standard $\la$-tableaux and the reduced decompositions of $w^\la$.
\end{enumerate}
\end{Proposition}

Dominant minuscule elements are known to be fully commutative, see e.g. \cite[Proposition 2.1]{S2}, and can be characterized in terms of their reduced expressions as follows.

\begin{Proposition}\label{PDM}{\rm \cite[Proposition 2.5]{S2}} 
If $w = r_{i_1} \dots r_{i_d} \in W$ is a reduced expression, then $w$ 
is dominant minuscule if and only if the following two conditions are satisfied:
\begin{enumerate}
\item[{\rm (i)}] between every pair of occurrences of a generator $r_i$ 
(with no other occurrences of $r_i$ in between) there are exactly two terms (possibly  equal to each other) that do not commute with $r_i$;
\item[{\rm (ii)}] the last occurrence of each generator $r_i$ is followed by at most 
one generator that does not commute with $r_i$.
\end{enumerate} 
\end{Proposition}

Now it is easy to see that in type $A_\infty$, skew shapes $\la$ with 
$w^\la$ dominant minuscule are (disjoint unions of) `straight' shapes in the usual sense, i.e. 
partitions drawn in the Russian notation. This motivates the following
definition. 
A skew shape $\la$ is a {\em straight shape} if $w^\la$ is dominant minuscule.  Proposition~\ref{PDM} yields the following explicit characterization of the straight shapes.

\begin{Lemma}
Let $\la$ be a configuration. Then $\la$ is a straight shape if and only if 
the following conditions are satisfied:
\begin{enumerate}
\item[{\rm (i)}] between every pair of beads $A$, $B$ on a runner $i$ 
(with no beads on the runner $i$ between $A$ and $B$) there are exactly two beads between $A$ and $B$, which lie on runners neighboring $i$ (possibly on the same runner);
\item[{\rm (ii)}] the bottom bead on a runner $i$ has at most 
one bead below it on runners neighboring $i$.
\end{enumerate}
\end{Lemma}

Peterson and Proctor have given  a hook-type formula for the number of standard
tableaux of a straight shape.
The proof of this hook formula, and generalizations of it,
can be found e.g. in Nakada in \cite{N2}. 
In view of Proposition~\ref{PSt}(iii)), the Peterson-Proctor hook
formula can be stated, in our context, as follows.

\begin{Theorem}
{\bf (Peterson-Proctor Hook Formula)}
Let $\la$ be a straight shape with $d$ beads. Using notation as in Theorem~\ref{tabconst},
the dimension of the corresponding representation of the Khovanov-Lauda algebra is
$$
\dim S(\lambda) = \mathrm{Card}(\mathcal{T}(\la)) 
= \frac{d!}{\prod_{\beta\in \Phi(w^\la)}\height(\beta)},
$$
where
$
\Phi(w):=\{\be\in \Phi_+\mid w^{-1}(\be)<0\},
$
and $\mathrm{Card}(\mathcal{T}(\la))$ is the number of standard tableaux of shape $\lambda$.
\end{Theorem}


\begin{thebibliography}{ABC}











\bibitem[BK]{BK}
J. Brundan and A. Kleshchev,
Blocks of cyclotomic Hecke algebras and Khovanov-Lauda algebras,
preprint, University of Oregon, 2008.

\bibitem[Ch]{Ch}
I. Cherednik, Special bases of irreducible representations of a degenerate affine Hecke algebra, {\em Functional Anal. Appl.} {\bf 20} (1986), 76--78.



\bibitem[F]{F} C.K. Fan, A Hecke algebra quotient and some combinatorial applications, {\em  J. Algebraic Combin.} {\bf 5} (1996), 175--189. 



\bibitem[Ka]{Kac}
V.G. Kac, {\em Infinite Dimensional Lie Algebras}, Camb. Univ. Press, 1990. 

\bibitem[KL1]{KL1}
M. Khovanov and A. Lauda,
A diagrammatic approach to categorification of quantum groups I;
{\tt arXiv:0803.4121}.

\bibitem[KL2]{KL2}
M. Khovanov and A. Lauda,
A diagrammatic approach to categorification of quantum groups II;
{\tt  arXiv:0804.2080}


\bibitem[K1]{KCS}
A. Kleshchev, Completely splittable representations of symmetric groups, {\em J. Algebra} {\bf 181} (1996), 584--592. 


\bibitem[K2]{Kbook}
A. Kleshchev, {\em Linear and Projective Representations of Symmetric Groups}, Cambridge University Press, 2005. 


\bibitem[KR]{KrR}
C. Kriloff and A. Ram, Representations of graded Hecke algebras, {\em Represent.  Theory} {\bf 6} (2002), 31--69.





\bibitem[Mac]{Mac}
I.G.\ Macdonald, {\em Symmetric functions and Hall polynomials}, Second edition,
Oxford University Press 1995.


\bibitem[Ma]{Ma}
A. Mathas, Seminormal forms and Gram determinants for cellular algebras, {\em J. reine angew. Math.} {\bf 619} (2008), 141--173.

\bibitem[Mo]{Mo}
A. Molev, 
Skew representations of twisted Yangians, 
{\em Selecta Math. (N.S.)} {\bf 12} (2006), 1--38. 

\bibitem[N1]{N1}
K. Nakada, A hook formula for the standard tableaux of a generalized shape, {\em RIMS Kokyouroku-bessatsu} {\bf B8} (2008), 55--62.

\bibitem[N2]{N2}
K. Nakada, Colored hook formula for a generalized Young diagram, {\em Osaka J. of Math.} {\bf 45} (2008), to appear.



\bibitem[O]{Ok}
A. Okounkov, Characters of Symmetric Groups,
{\em Minicourse at the Introductory Workshop on Combinatorial Representation Theory},
Mathematical Sciences Research Institute, Berkeley,
22-26 January 2008, 
http://www.msri.org/communications/vmath/VMathVideos/VideoInfo/3561/ show\_video


\bibitem[OV]{OV}
A. Okounkov and A. Vershik, A new approach to representation theory of symmetric groups. {\em Selecta Math. (N.S.)} {\bf 2} (1996), 581--605. 



\bibitem[P1]{P1}
R.A. Proctor, Minuscule elements of Weyl groups, the numbers game, and $d$-complete posets, {\em J. Algebra} {\bf 213} (1999), 272--303.

\bibitem[P2]{P2}
R.A. Proctor, Dynkin diagram classification of $\lambda$-minuscule Bruhat lattices and $d$-complete posets, {\em J. Algebraic Combin.} {\bf 9} (1999), 61--94.

\bibitem[Ro]{Ro}
R. Rouquier, $2$-Kac-Moody algebras, {\em in preparation}. 

\bibitem[Ru]{R}
O. Ruff, Completely splittable representations of symmetric groups and affine Hecke algebras, {\em J. Algebra} {\bf 305} (2006), 1197--1211.

\bibitem[S1]{S1}
J.R. Stembridge, On the fully commutative elements of Coxeter groups, {\em J. Algebraic Combin.} {\bf 5} (1996), 353--385.



\bibitem[S2]{S2}
J.R. Stembridge, Minuscule elements of Weyl groups, {\em J. Algebra} {\bf 235} (2001), 722--743.



\bibitem[VK]{VK}
A.M. Vershik, A. M. and S.V. Kerov, Asymptotic behavior of the Plancherel measure of the symmetric group and the limit form of Young tableaux. (in Russian) {\em Dokl. Akad. Nauk SSSR}  {\bf 233} (1977), no. 6, 1024--1027. 

\bibitem[V]{V}
G.X. Viennot, Heaps of pieces I: Basic definitions and combinatorial lemmas, in {\em Combinatoire \'Enum\'erative}, 
G. Labelle and P. Leroux (Eds.), pp. 321-350, Lect. Notes in Math. Vol. 1234, Springer-Verlag, 1985.

\end{thebibliography}
\end{document}